\documentclass[11pt,a4paper]{article}

\usepackage[T1]{fontenc}
\usepackage[utf8]{inputenc}
\usepackage[margin=1in]{geometry}
\usepackage{amsmath,amssymb,amsthm,mathtools}
\usepackage{enumitem}
\usepackage{microtype}
\usepackage[hidelinks]{hyperref}

\allowdisplaybreaks
\setlist[enumerate,1]{label=\textup{(\roman*)},leftmargin=2.4em}

\theoremstyle{plain}
\newtheorem{theorem}{Theorem}[section]
\newtheorem{lemma}[theorem]{Lemma}
\newtheorem{proposition}[theorem]{Proposition}
\newtheorem{corollary}[theorem]{Corollary}
\newtheorem{conjecture}[theorem]{Conjecture}

\theoremstyle{definition}
\newtheorem{definition}[theorem]{Definition}

\theoremstyle{remark}
\newtheorem{remark}[theorem]{Remark}

\newcommand{\tr}{\operatorname{tr}}
\newcommand{\homc}{\operatorname{hom}}
\newcommand{\ind}{\operatorname{ind}}
\newcommand{\Forb}{\operatorname{Forb}}
\newcommand{\Split}{\mathsf{Split}}
\newcommand{\R}{\mathbb{R}}
\newcommand{\Z}{\mathbb{Z}}
\newcommand{\one}{\mathbf{1}}

\begin{document}

\title{Generalized Spectral Closedness and the Nonexistence of Walk-Realizable Supporters for Split Graphs}
\author{Yanyang Li}

\date{}

\maketitle

\begin{abstract}
Let $\mathcal F=\{2K_2,C_4,C_5\}$, so that $\operatorname{Forb}(\mathcal F)$ is the class of split graphs. Resolving a conjecture of Wang and Tang, we prove that the class of split graphs is generalized spectrally closed and that it admits no walk-realizable $\mathcal F$-supporter of any finite order.
\end{abstract}

\noindent\textbf{Keywords:} split graph; generalized spectrum; patterned closed walk; graph homomorphism; \(2\)-SAT.

\noindent\textbf{Mathematics Subject Classification (2020):} 05C50; 05C75.

\section{Introduction}

All graphs in this paper are finite, simple, and nonempty. We write \(K_r\), \(P_r\), and \(C_r\) for the complete graph, the path, and the cycle on \(r\) vertices, respectively, and \(2K_2=K_2\sqcup K_2\). For a graph \(G\), its complement is denoted by \(\overline G\), its adjacency matrix by \(A(G)\), and the subgraph induced by \(U\subseteq V(G)\) by \(G[U]\). If \(G\) has order \(n\), then \(I_n\), \(J_n\), and \(\one_n\) denote the identity matrix, the all-ones matrix, and the all-ones column vector of the indicated size; subscripts are omitted when the size is clear.

For a family \(\mathcal F\) of graphs, let \(\Forb(\mathcal F)\) be the class of graphs containing no member of \(\mathcal F\) as an induced subgraph. A graph is \emph{split} if its vertex set is the disjoint union of a clique and an independent set. The forbidden-induced-subgraph theorem of F\"oldes and Hammer gives \(\Split=\Forb(2K_2,C_4,C_5)\); see \cite{FoldesHammer}. Throughout the paper we therefore set \(\mathcal F:=\{2K_2,C_4,C_5\}\).

Two graphs are \emph{generalized cospectral} if their adjacency matrices are cospectral and the adjacency matrices of their complements are cospectral. A graph class \(\mathcal C\) is \emph{generalized spectrally closed} if every graph generalized cospectral with a member of \(\mathcal C\) also belongs to \(\mathcal C\).

Wang and Tang \cite{WangTang} introduced patterned closed walks to study generalized spectral closedness of hereditary graph classes. For a nonempty binary word \(\beta=(\beta_1,\ldots,\beta_k)\), put \(M_1(G):=A(G)\) and \(M_0(G):=A(\overline G)\), and define
\begin{equation}\label{eq:beta-count}
 n_\beta(G):=\tr\!\left(M_{\beta_1}(G)\cdots M_{\beta_k}(G)\right).
\end{equation}
Equivalently, \(n_\beta(G)\) counts the closed vertex sequences whose \(i\)-th step is an edge of \(G\) when \(\beta_i=1\) and an edge of \(\overline G\) when \(\beta_i=0\). Wang and Tang called these \(\beta\)-closed walks and proved that every \(n_\beta\) is invariant under generalized cospectrality \cite[Lemma~2.2]{WangTang}.

For \(\ell\geq1\), let \(\mathcal U_{\leq\ell}\) be the set of isomorphism classes of nonempty graphs on at most \(\ell\) vertices. We use the normalization \(\ind(H,G):=\bigl|\{U\subseteq V(G):G[U]\cong H\}\bigr|\), so that \(\ind(H,H)=1\).

\begin{definition}[Wang--Tang]\label{def:supporter}
Let \(\ell\geq5\). An \(\mathcal F\)-\emph{supporter of order \(\ell\)} is a graph parameter of the form
\[
 \Phi(G)=\sum_{H\in\mathcal U_{\leq\ell}}d_H\ind(H,G)
\]
whose coefficients satisfy \(d_H\geq0\) for all \(H\), \(d_F>0\) for every \(F\in\mathcal F\), and \(d_H=0\) whenever \(H\in\Forb(\mathcal F)\). It is \emph{walk-realizable} if it is a real linear combination of the parameters \(n_\beta\) with \(1\leq |\beta|\leq\ell\).
\end{definition}

Wang and Tang formulate walk-realizability using one representative from each equivalence class of binary words under cyclic shifts and reversal. This is equivalent to Definition~\ref{def:supporter}, because the trace in \eqref{eq:beta-count} is unchanged by cyclic shifts and, since \(A(G)\) and \(A(\overline G)\) are symmetric, by reversal of the word.

The conditions in Definition~\ref{def:supporter} imply that \(\Phi(G)=0\) exactly when \(G\in\Forb(\mathcal F)\): every induced subgraph of an \(\mathcal F\)-free graph is \(\mathcal F\)-free, whereas a graph outside \(\Forb(\mathcal F)\) contains some \(F\in\mathcal F\) as an induced subgraph. Thus a walk-realizable supporter is a finite, nonnegative induced-subgraph certificate encoded by generalized spectral invariants. Wang and Tang found no walk-realizable \(\mathcal F\)-supporter for split graphs up to order \(7\), and found no counterexample to generalized spectral closedness in an exhaustive search over small graphs. This led them to the following conjecture \cite[Conjecture~4.8]{WangTang}.

\begin{conjecture}[Wang--Tang]\label{conj:WangTang}
Let \(\mathcal F=\{2K_2,C_4,C_5\}\). The class
\(\Forb(\mathcal F)\) of split graphs is generalized spectrally closed, but it admits no walk-realizable \(\mathcal F\)-supporter of order \(\ell\) for any integer \(\ell\geq5\).
\end{conjecture}

We resolve Conjecture~\ref{conj:WangTang} in full.

\begin{theorem}\label{thm:main}
For \(\mathcal F=\{2K_2,C_4,C_5\}\), the following statements hold.
\begin{enumerate}
 \item The class \(\Forb(\mathcal F)=\Split\) is generalized spectrally closed.
 \item For no integer \(\ell\geq5\) does there exist a walk-realizable \(\mathcal F\)-supporter of order \(\ell\).
\end{enumerate}
\end{theorem}

The two parts are proved by different structural arguments. For the first, we construct a \(2|V(G)|\times 2|V(G)|\) nonnegative matrix \(T_G\) from a canonical \(2\)-SAT formula. We show that \(G\) is split precisely when \(\chi_{T_G}\) is a square. The literal-complementation involution pairs the strongly connected components of the implication digraph. A satisfiable instance has only paired components, whereas an unsatisfiable instance has a unique self-dual component; the Perron root of that component occurs with odd multiplicity.

For the second part, let \(\mathcal T:=\operatorname{span}_{\R}\{n_\beta:\beta\text{ is a nonempty binary word}\}\), where, as usual, \(\operatorname{span}\) means finite linear span. We prove the stronger statement that \(\Phi\in\mathcal T\) and \(\Phi(S)=0\) for every split graph \(S\) together imply \(\Phi\equiv0\). A normal form reduces \(\Phi\) to closed-walk counts \(\tr(A^r)\) and total-walk counts \(\one^{\top}A^r\one\). We then evaluate the normal form on a two-parameter family of split blow-ups and eliminate its coefficients by descending in the \(a\)- and \(b\)-degrees. The final separation step uses the linear independence of graph-homomorphism counts \cite[Corollary~5.45]{LovaszBook}.

\section{A square criterion for split graphs}

Let \(G\) be a graph on vertex set \([n]:=\{1,\ldots,n\}\), write \(A=A(G)\) and \(\overline A=A(\overline G)=J_n-I_n-A\), and define
\begin{equation}\label{eq:TG}
 T_G:=\begin{pmatrix}0&\overline A\\ A&0\end{pmatrix}.
\end{equation}

\begin{theorem}\label{thm:square-criterion}
A graph \(G\) is split if and only if \(\det(zI_{2n}-T_G)\) is a square in \(\Z[z]\).
\end{theorem}

\begin{proof}
Introduce one Boolean variable \(x_i\) for each \(i\in[n]\), where \(x_i=1\) means that \(i\) is assigned to the clique part. Consider the \(2\)-CNF formula
\begin{equation}\label{eq:2sat-formula}
 \Psi_G:=
 \bigwedge_{\substack{1\leq i<j\leq n\\ ij\in E(G)}}(x_i\vee x_j)
 \ \wedge\!
 \bigwedge_{\substack{1\leq i<j\leq n\\ ij\notin E(G)}}(\neg x_i\vee\neg x_j).
\end{equation}
In \eqref{eq:2sat-formula}, the edge clauses prevent an edge from having both endpoints in the independent part, and the nonedge clauses prevent a nonedge from having both endpoints in the clique part. Hence
\begin{equation}\label{eq:split-sat}
 G\text{ is split}\quad\Longleftrightarrow\quad \Psi_G\text{ is satisfiable}.
\end{equation}

Let \(\mathcal I_G\) be the implication digraph of \(\Psi_G\), with the literals ordered as \(x_1,\ldots,x_n,\neg x_1,\ldots,\neg x_n\). We use the row-source convention for adjacency matrices. If \(ij\in E(G)\), then the clause \(x_i\vee x_j\) gives the arcs \(\neg x_i\to x_j\) and \(\neg x_j\to x_i\). If \(ij\notin E(G)\), then \(\neg x_i\vee\neg x_j\) gives \(x_i\to\neg x_j\) and \(x_j\to\neg x_i\). Therefore the adjacency matrix of \(\mathcal I_G\) is exactly \(T_G\).

Let \(\tau\) be literal complementation: \(\tau(x_i)=\neg x_i\) and \(\tau(\neg x_i)=x_i\). The implication digraph has the contrapositive symmetry
\begin{equation}\label{eq:contrapositive}
 u\to v\quad\Longleftrightarrow\quad \tau(v)\to\tau(u).
\end{equation}
By \eqref{eq:contrapositive}, if \(C\) is a strongly connected component, then \(C^*:=\tau(C)\) is also a strongly connected component. Under the bijection \(u\mapsto\tau(u)\), the digraph induced by \(C^*\) is the transpose of the digraph induced by \(C\). Thus, writing \(T_G[C]\) for the principal submatrix indexed by \(C\),
\begin{equation}\label{eq:paired-charpoly}
 \chi_{T_G[C^*]}(z)=\chi_{T_G[C]}(z).
\end{equation}
Call \(C\) \emph{self-dual} if \(C=C^*\).

The strongly connected component criterion for \(2\)-SAT says that \(\Psi_G\) is unsatisfiable precisely when some pair \(x_i,\neg x_i\) lies in one strongly connected component \cite{AspvallPlassTarjan,AspvallErratum}. Such a component is self-dual. Conversely, every self-dual component contains \(u\) and \(\tau(u)\) for each \(u\) in it, and hence contains a complementary pair. Therefore
\begin{equation}\label{eq:unsat-selfdual}
 \Psi_G\text{ is unsatisfiable}\quad\Longleftrightarrow\quad
 \mathcal I_G\text{ has a self-dual strongly connected component}.
\end{equation}

We next show that there is at most one self-dual component. Suppose that \(C\) and \(D\) are distinct self-dual components. Since a self-dual component contains a complementary pair, choose distinct indices \(i,j\) with \(x_i,\neg x_i\in C\) and \(x_j,\neg x_j\in D\). If \(ij\in E(G)\), then \(\neg x_i\to x_j\) and \(\neg x_j\to x_i\), so the condensation digraph has an arc from \(C\) to \(D\) and an arc from \(D\) to \(C\). If \(ij\notin E(G)\), the arcs \(x_i\to\neg x_j\) and \(x_j\to\neg x_i\) give the same conclusion. Both alternatives contradict the acyclicity of the condensation digraph. Hence
\begin{equation}\label{eq:unique-selfdual}
 \mathcal I_G\text{ has at most one self-dual strongly connected component}.
\end{equation}

After ordering the strongly connected components topologically, \(T_G\) is block upper triangular. Its characteristic polynomial therefore factors as
\begin{equation}\label{eq:scc-factorization}
 \chi_{T_G}(z)=\prod_C\chi_{T_G[C]}(z),
\end{equation}
where the product is over all strongly connected components.

Suppose first that \(G\) is split. By \eqref{eq:split-sat} and \eqref{eq:unsat-selfdual}, there is no self-dual component. The components form disjoint pairs \(C,C^*\), and \eqref{eq:paired-charpoly} and \eqref{eq:scc-factorization} show that every pair contributes \(\chi_{T_G[C]}(z)^2\). Thus \(\chi_{T_G}\) is a square in \(\Z[z]\).

Conversely, suppose that \(G\) is not split. Then \(\Psi_G\) is unsatisfiable, and \eqref{eq:unsat-selfdual}--\eqref{eq:unique-selfdual} give a unique self-dual component \(C_0\). All other components occur in dual pairs, so
\begin{equation}\label{eq:selfdual-factor}
 \chi_{T_G}(z)=\chi_{T_G[C_0]}(z)R(z)^2
\end{equation}
for some \(R\in\Z[z]\). The matrix \(T_G[C_0]\) is nonnegative and irreducible. It has order at least two and its strongly connected digraph contains a directed cycle, so its spectral radius \(\rho>0\). By the Perron--Frobenius theorem, \(\rho\) is an algebraically simple root of \(\chi_{T_G[C_0]}\). Its multiplicity in the right-hand side of \eqref{eq:selfdual-factor} is therefore \(1+2m\) for some \(m\geq0\), which is odd. A square polynomial has only roots of even multiplicity, so \(\chi_{T_G}\) is not a square. This proves the converse.
\end{proof}

\section{Generalized spectral closedness}

We use the following standard characterization of Johnson and Newman \cite{JohnsonNewman}. An orthogonal matrix \(Q\) is called \emph{regular} if \(Q\one=\one\), equivalently, if every row sum is one.

\begin{theorem}[Johnson--Newman]\label{thm:JN}
Let \(G\) and \(H\) be graphs on \(n\) vertices. They are generalized cospectral if and only if there is a regular orthogonal matrix \(Q\) such that
\[
 Q^{\top}A(G)Q=A(H).
\]
For such a matrix one also has \(Q^{\top}A(\overline G)Q=A(\overline H)\).
\end{theorem}

The last assertion follows from \(A(\overline G)=J-I-A(G)\) and \(Q^{\top}JQ=J\). It also immediately yields the patterned-walk invariance proved by Wang and Tang \cite[Lemma~2.2]{WangTang}.

\begin{proposition}[Wang--Tang]\label{prop:beta-invariance}
If \(G\) and \(H\) are generalized cospectral, then \(n_\beta(G)=n_\beta(H)\) for every nonempty binary word \(\beta\).
\end{proposition}

\begin{proof}
By Theorem~\ref{thm:JN}, the two matrices \(M_0(G),M_1(G)\) are simultaneously conjugated to \(M_0(H),M_1(H)\) by one orthogonal matrix. Insert \(QQ^{\top}=I\) between consecutive factors in \eqref{eq:beta-count} and use invariance of trace under similarity.
\end{proof}

\begin{theorem}\label{thm:generalized-closure}
The class of split graphs is generalized spectrally closed.
\end{theorem}

\begin{proof}
Let \(G\) and \(H\) be generalized cospectral, and let \(Q\) be supplied by Theorem~\ref{thm:JN}. With \(T_G,T_H\) as in \eqref{eq:TG} and
\(\widehat Q:=\operatorname{diag}(Q,Q)\),
\[
 \widehat Q^{\top}T_G\widehat Q
 =\begin{pmatrix}
 0&Q^{\top}A(\overline G)Q\\
 Q^{\top}A(G)Q&0
 \end{pmatrix}
 =T_H.
\]
Thus \(T_G\) and \(T_H\) have the same characteristic polynomial. Theorem~\ref{thm:square-criterion} implies that \(G\) is split if and only if \(H\) is split.
\end{proof}

Theorem~\ref{thm:generalized-closure} proves Theorem~\ref{thm:main}\textup{(i)}.

\section{A normal form for trace-word parameters}

For a graph \(G\) with adjacency matrix \(A\), define \(p_r(G):=\tr(A^r)\) and \(q_r(G):=\one^{\top}A^r\one\) for \(r\geq0\). Thus \(p_r\) counts closed walks of length \(r\), whereas \(q_r\) counts all walks of length \(r\) with arbitrary initial and terminal vertices. In particular, \(p_0=q_0=|V(G)|\), \(p_1=0\), and \(p_2=q_1\).

Write \(\homc(F,G)\) for the number of graph homomorphisms from \(F\) to \(G\). If \(\lambda=(\lambda_1,\ldots,\lambda_s)\) is a finite nonempty multiset of nonnegative integers, set \(Q_\lambda(G):=\prod_{i=1}^s q_{\lambda_i}(G)\). We shall use the following classical independence theorem; see
\cite[Corollary~5.45]{LovaszBook}.

\begin{theorem}[Lov\'asz]\label{thm:lovasz-independence}
For pairwise nonisomorphic finite simple graphs \(F_1,\ldots,F_t\), the graph parameters \(G\mapsto\homc(F_i,G)\) are linearly independent over \(\R\).
\end{theorem}

\begin{lemma}\label{lem:normal-form}
Every \(\Phi\in\mathcal T\) has a unique representation
\begin{equation*}
 \Phi(G)=\sum_{r=3}^{R}c_rp_r(G)+\sum_{\lambda\in\Lambda}d_\lambda Q_\lambda(G),
\end{equation*}
where \(R\) is finite and \(\Lambda\) is a finite set of finite nonempty multisets of nonnegative integers.
\end{lemma}

\begin{proof}
Write \(\overline A=J-I-A\), where \(J=\one\one^{\top}\), and expand each trace word \(n_\beta\) into traces of words in \(A,J,I\). After deleting factors \(I\), a term with no occurrence of \(J\) is \(\tr(A^r)=p_r\). The cases \(r=0,1,2\) are absorbed by \(p_0=q_0\), \(p_1=0\), and \(p_2=q_1\).

If a term has \(s\geq1\) occurrences of \(J\), cyclicity of trace puts it in the form \(\tr(A^{r_1}JA^{r_2}J\cdots A^{r_s}J)\). Since \(J=\one\one^{\top}\),
\begin{equation*}
 \tr(A^{r_1}JA^{r_2}J\cdots A^{r_s}J)
 =\prod_{i=1}^s\one^{\top}A^{r_i}\one
 =\prod_{i=1}^s q_{r_i}.
\end{equation*}
This proves existence.

For uniqueness, observe that \(p_r(G)=\homc(C_r,G)\) for \(r\geq3\), while \(q_r(G)=\homc(P_{r+1},G)\), with \(P_1\) the one-vertex graph. Since homomorphism counts multiply over disjoint unions,
\[
 Q_\lambda(G)=\homc\!\left(\bigsqcup_{i=1}^sP_{\lambda_i+1},G\right).
\]
Distinct multisets \(\lambda\) produce pairwise nonisomorphic linear forests, and no such forest is isomorphic to a cycle \(C_r\) with \(r\geq3\). Theorem~\ref{thm:lovasz-independence} therefore gives uniqueness.
\end{proof}

\section{Two-parameter split blow-ups}

Fix an arbitrary graph \(G\) on \(m\) vertices, put \(A=A(G)\), and set \(X:=I_m+A\). For positive integers \(a,b\), construct a graph \(S_{a,b}(G)\) as follows. For each \(i\in[m]\), take a block \(C_i\) of \(a\) vertices and a block \(D_i\) of \(b\) vertices. Make \(\bigcup_iC_i\) a clique and \(\bigcup_iD_i\) an independent set, and join \(C_i\) completely to \(D_j\) precisely when \(X_{ij}=1\). By construction, \(S_{a,b}(G)\) is split.

For any square matrix \(Y\), write \(p_r(Y):=\tr(Y^r)\) and \(q_r(Y):=\one^{\top}Y^r\one\), where the all-ones vector has the dimension of \(Y\).

\begin{lemma}\label{lem:quotient}
Let \(\one_{C_i}\) and \(\one_{D_i}\) be the characteristic vectors of the indicated blocks, viewed in \(\R^{V(S_{a,b}(G))}\), and put \(e_i^C:=a^{-1/2}\one_{C_i}\) and \(e_i^D:=b^{-1/2}\one_{D_i}\). The vectors
\[
 e_1^C,\ldots,e_m^C,e_1^D,\ldots,e_m^D
\]
form an orthonormal basis of the block-constant subspace. In this basis, the adjacency operator of \(S_{a,b}(G)\) is represented by
\begin{equation}\label{eq:symmetric-quotient}
 \mathcal B_{a,b}(X):=
 \begin{pmatrix}
 aJ_m-I_m&\sqrt{ab}\,X\\
 \sqrt{ab}\,X&0
 \end{pmatrix}.
\end{equation}
The orthogonal complement of that subspace contributes the eigenvalue \(-1\) with multiplicity \(m(a-1)\) and the eigenvalue \(0\) with multiplicity \(m(b-1)\). Consequently, for \(r\geq1\),
\begin{equation}\label{eq:p-quotient}
 p_r(S_{a,b}(G))=\tr\!\left(\mathcal B_{a,b}(X)^r\right)+m(a-1)(-1)^r,
\end{equation}
and, for \(r\geq0\),
\begin{equation*}
 q_r(S_{a,b}(G))=z_{a,b}^{\top}\mathcal B_{a,b}(X)^rz_{a,b},
 \qquad z_{a,b}:=\binom{\sqrt a\,\one_m}{\sqrt b\,\one_m}.
\end{equation*}
\end{lemma}

\begin{proof}
The matrix entry between \(e_i^C\) and \(e_j^C\) is \(a\) for \(i\neq j\) and \(a-1\) for \(i=j\), giving \(aJ_m-I_m\). The entry between \(e_i^C\) and \(e_j^D\) is \(\sqrt{ab}\,X_{ij}\), giving the off-diagonal blocks.

A vector supported on one \(C_i\) and summing to zero is multiplied by \(-1\); all interactions with other blocks vanish because they depend only on the block sum. These vectors have total dimension \(m(a-1)\). Similarly, vectors supported on one \(D_i\) and summing to zero are annihilated and have total dimension \(m(b-1)\). Finally, the all-ones vector has coordinates \(z_{a,b}\) in the displayed orthonormal basis. The two formulas follow.
\end{proof}

\begin{lemma}\label{lem:polynomial-dependence}
For every fixed \(G\) and every \(r\geq0\), both \(p_r(S_{a,b}(G))\) and \(q_r(S_{a,b}(G))\), initially defined for positive integers \(a,b\), are restrictions of polynomials in \(\R[a,b]\). The same is true of \(\Phi(S_{a,b}(G))\) for every \(\Phi\in\mathcal T\).
\end{lemma}

\begin{proof}
On block-constant vectors represented by their constant values on the blocks, the same adjacency operator has the unnormalized quotient
\[
 \mathcal Q_{a,b}(X):=
 \begin{pmatrix}
 aJ_m-I_m&bX\\
 aX&0
 \end{pmatrix}.
\]
The matrix \(\mathcal Q_{a,b}(X)\) has polynomial entries and is similar to \(\mathcal B_{a,b}(X)\). Indeed,
\[
 D:=\operatorname{diag}(\sqrt a I_m,\sqrt b I_m),
 \qquad
 \mathcal B_{a,b}(X)=D\mathcal Q_{a,b}(X)D^{-1}.
\]
Hence, for \(r\geq1\),
\[
 p_r(S_{a,b}(G))=\tr(\mathcal Q_{a,b}(X)^r)+m(a-1)(-1)^r.
\]
Moreover,
\[
 q_r(S_{a,b}(G))=
 \begin{pmatrix}a\one_m^{\top}&b\one_m^{\top}\end{pmatrix}
 \mathcal Q_{a,b}(X)^r
 \binom{\one_m}{\one_m}.
\]
These expressions are polynomial in \(a,b\); the case \(r=0\) is \(m(a+b)\). The final assertion follows from the normal form in Lemma~\ref{lem:normal-form}.
\end{proof}

For a polynomial \(P(a,b)\), let \([a^ub^v]P\) denote the coefficient of \(a^ub^v\). For \(r\geq0\), define
\begin{equation}\label{eq:R-def}
 R_{2r}(X):=q_{2r}(X),\qquad
 R_{2r+1}(X):=2q_{2r+1}(X)+\sum_{h=0}^{r-1}q_{2h+1}(X)q_{2r-2h-1}(X),
\end{equation}
where the sum is empty for \(r=0\).

\begin{lemma}\label{lem:leading-coefficients}
For every \(r\geq1\),
\begin{align}
 [b^r]p_{2r}(S_{a,b}(G))&=2a^rp_{2r}(X),\label{eq:p-even-leading}\\
 [b^r]p_{2r+1}(S_{a,b}(G))&=(2r+1)a^r\bigl(aq_{2r}(X)-p_{2r}(X)\bigr).\label{eq:p-odd-leading}
\end{align}
For every \(r\geq0\),
\begin{align}
 [b^{r+1}]q_{2r}(S_{a,b}(G))&=a^rq_{2r}(X),\label{eq:q-even-leading}\\
 [b^{r+1}]q_{2r+1}(S_{a,b}(G))&=a^{r+1}R_{2r+1}(X)-ra^rq_{2r}(X).\label{eq:q-odd-leading}
\end{align}
In each formula, the displayed power of \(b\) is the largest one that occurs.
\end{lemma}

\begin{proof}
With \(\mathcal B_{a,b}(X)\) as in~\eqref{eq:symmetric-quotient}, put \(C:=aJ_m-I_m\) and \(t:=\sqrt{ab}\). In the block expansion of a matrix power, a passage between the clique side and the independent side contributes \(t\), while a step remaining on the independent side contributes zero.

For \(p_{2r}\), the largest possible number of passages is \(2r\), attained by alternating sides throughout. There are two choices of starting side, so the contribution of maximal \(b\)-degree is \(2(ab)^r\tr(X^{2r})\), proving \eqref{eq:p-even-leading}.

For \(p_{2r+1}\), a term of maximal \(b\)-degree has \(2r\) passages and exactly one clique-side step \(C\). That step has \(2r+1\) cyclic positions, giving \((2r+1)(ab)^r\tr(CX^{2r})\). Since \(\tr(CX^{2r})=a\tr(J_mX^{2r})-\tr(X^{2r})=aq_{2r}(X)-p_{2r}(X)\), this is \eqref{eq:p-odd-leading}. The correction term in \eqref{eq:p-quotient} has \(b\)-degree zero.

For \(q_{2r}\), maximal \(b\)-degree is obtained by starting and ending on the independent side and alternating at every step. The endpoint factors contribute \(b\), while the \(2r\) passages contribute \((ab)^r\), proving \eqref{eq:q-even-leading}.

For \(q_{2r+1}\), there are exactly two types of terms of \(b\)-degree \(r+1\). If the endpoints lie on opposite sides and every step alternates, the two endpoint orientations contribute \(2(ab)^{r+1}q_{2r+1}(X)\). If both endpoints lie on the independent side, there must be \(2r\) passages and one clique-side step \(C\). These terms contribute
\[
 a^rb^{r+1}\sum_{h=0}^{r-1}
 \one^{\top}X^{2h+1}CX^{2r-2h-1}\one.
\]
For each \(h\),
\[
 \one^{\top}X^{2h+1}CX^{2r-2h-1}\one
 =a q_{2h+1}(X)q_{2r-2h-1}(X)-q_{2r}(X).
\]
Adding both types gives \eqref{eq:q-odd-leading}. The alternating description also shows in each case that no larger power of \(b\) is possible.
\end{proof}

\section{Vanishing on split graphs forces universal vanishing}

For a finite nonempty multiset \(\lambda=(\lambda_1,\ldots,\lambda_s)\), define
\begin{equation}\label{eq:DETheta}
 D(\lambda):=\sum_{i=1}^s\left(\left\lfloor\frac{\lambda_i}{2}\right\rfloor+1\right),
 \qquad
 E(\lambda):=\sum_{i=1}^s\left\lceil\frac{\lambda_i}{2}\right\rceil,
 \qquad
 \Theta_\lambda(X):=\prod_{i=1}^sR_{\lambda_i}(X).
\end{equation}
Notice that \(E(\lambda)\leq D(\lambda)\).

\begin{lemma}\label{lem:Q-leading}
With the notation of \eqref{eq:DETheta}, the largest \(b\)-degree in
\(Q_\lambda(S_{a,b}(G))\) is \(D(\lambda)\). Within the coefficient of \(b^{D(\lambda)}\), the largest \(a\)-degree is \(E(\lambda)\), and
\begin{equation*}
 [a^{E(\lambda)}b^{D(\lambda)}]Q_\lambda(S_{a,b}(G))=\Theta_\lambda(X).
\end{equation*}
\end{lemma}

\begin{proof}
Apply \eqref{eq:q-even-leading}--\eqref{eq:q-odd-leading} to each
factor \(q_{\lambda_i}(S_{a,b}(G))\). To obtain the largest total
\(b\)-degree in the product, every factor must contribute its own
largest \(b\)-degree. In the resulting product of leading
\(b\)-coefficients, every factor has largest \(a\)-degree
\(\lceil\lambda_i/2\rceil\) and, by \eqref{eq:R-def}, corresponding
coefficient \(R_{\lambda_i}(X)\). Multiplication gives the assertion.
\end{proof}

\begin{lemma}\label{lem:theta-independence}
As \(\lambda\) ranges over finite nonempty multisets of nonnegative integers, the graph parameters
\[
 G\longmapsto\Theta_\lambda(I+A(G))
\]
are linearly independent over \(\R\).
\end{lemma}

\begin{proof}
First work in the polynomial ring
\(\R[y_0,y_1,\ldots]\). Define an endomorphism \(T\) by
\[
T(y_{2r})=y_{2r},\qquad
T(y_{2r+1})
=
2y_{2r+1}
+\sum_{h=0}^{r-1}
y_{2h+1}y_{2r-2h-1}
\qquad (r\geq0).
\]
On every finite-variable subring, \(T\) is a triangular polynomial
automorphism: \(T(y_j)\) is a nonzero scalar multiple of \(y_j\) plus
a polynomial in variables of smaller index.

Next define
\[
B(y_r):=\sum_{j=0}^r\binom{r}{j}y_j
\qquad (r\geq0).
\]
Again, \(B\) is triangular with diagonal coefficient one, and hence is
a polynomial automorphism on every finite-variable subring. Therefore
the images under \(B\circ T\) of distinct monomials are linearly
independent as formal polynomials.

For every graph \(G\),
\[
q_r(I+A(G))
=
\sum_{j=0}^r\binom{r}{j}q_j(A(G)).
\]
Consequently, for every finite nonempty multiset
\(\lambda=(\lambda_1,\ldots,\lambda_s)\),
\[
\Theta_\lambda(I+A(G))
=
(B\circ T)\!\left(\prod_{i=1}^s y_{\lambda_i}\right)
\Big|_{y_j=q_j(A(G))}.
\]
It remains to show that distinct monomials in the parameters
\(q_j(A(G))\) are linearly independent as graph parameters. But
\[
\prod_{i=1}^s q_{\lambda_i}(A(G))
=
\homc\!\left(
\bigsqcup_{i=1}^s P_{\lambda_i+1},G
\right),
\]
and distinct multisets \(\lambda\) give pairwise nonisomorphic linear
forests. Theorem~\ref{thm:lovasz-independence} therefore gives the
required linear independence.
\end{proof}

\begin{theorem}\label{thm:split-testing}
Let \(\Phi\in\mathcal T\). If \(\Phi(S)=0\) for every split graph \(S\), then \(\Phi(G)=0\) for every graph \(G\).
\end{theorem}

\begin{proof}
Write \(\Phi\) in the unique normal form
\begin{equation}\label{eq:Phi-normal}
 \Phi(G)=\sum_{r=3}^{R}c_rp_r(G)+\sum_{\lambda\in\Lambda}d_\lambda Q_\lambda(G).
\end{equation}
We interpret \(c_r=0\) whenever \(r\notin\{3,\ldots,R\}\). Fix an arbitrary graph \(G\). Since every \(S_{a,b}(G)\) is split, \(\Phi(S_{a,b}(G))=0\) for all positive integers \(a,b\). By Lemma~\ref{lem:polynomial-dependence}, this is a polynomial in \(a,b\). A polynomial that vanishes on \(\Z_{>0}^2\) is identically zero: first fix \(b\) and vary \(a\), then apply the same argument to each coefficient as a polynomial in \(b\). Thus
\begin{equation}\label{eq:blowup-zero-polynomial}
 \Phi(S_{a,b}(G))=0\qquad\text{in }\R[a,b].
\end{equation}
Because the base graph \(G\) was arbitrary, every coefficient identity extracted below is an identity of graph parameters.

Suppose that the normal form is nonzero. Define \(D\) by
\begin{equation*}
 D:=\max\!\left(
 \left\{\left\lfloor\frac r2\right\rfloor:c_r\neq0\right\}
 \cup
 \{D(\lambda):d_\lambda\neq0\}
 \right).
\end{equation*}
This is the largest of the leading \(b\)-degrees assigned by Lemmas~\ref{lem:leading-coefficients} and~\ref{lem:Q-leading}; it is well defined and satisfies \(D\geq1\).

The coefficient of \(a^{D+1}b^D\) in \eqref{eq:blowup-zero-polynomial} can arise only from the term \(c_{2D+1}p_{2D+1}\). By \eqref{eq:p-odd-leading}, it equals \((2D+1)c_{2D+1}q_{2D}(X)\). Since \(q_{2D}(I+A(K_1))=1\), this graph parameter is not identically zero, and hence
\begin{equation}\label{eq:odd-eliminated}
 c_{2D+1}=0.
\end{equation}

Recall our convention that \(c_2=0\). After \eqref{eq:odd-eliminated}, the coefficient of \(a^Db^D\) is
\begin{equation}\label{eq:aD-bD}
 2c_{2D}p_{2D}(X)
 +\sum_{\substack{\lambda\in\Lambda\\D(\lambda)=D,\ E(\lambda)=D}}
 d_\lambda\Theta_\lambda(X)=0.
\end{equation}
When \(D=1\), the first term is absent because \(c_2=0\). Suppose \(D\geq2\). Since \(I\) commutes with \(A\),
\[
 p_{2D}(I+A)=\sum_{j=0}^{2D}\binom{2D}{j}p_j(A).
\]
The summand \(p_{2D}(A)=\homc(C_{2D},G)\) occurs with coefficient one. Every \(\Theta_\lambda(I+A)\) is a polynomial in the parameters \(q_j(A)\), hence a linear combination of homomorphism counts of linear forests. No linear forest is isomorphic to \(C_{2D}\). By Theorem~\ref{thm:lovasz-independence}, the coefficient of \(\homc(C_{2D},G)\) in \eqref{eq:aD-bD} must vanish, so \(c_{2D}=0\). Lemma~\ref{lem:theta-independence} then gives
\begin{equation*}
 d_\lambda=0\qquad\text{whenever }D(\lambda)=E(\lambda)=D.
\end{equation*}
For \(D=1\), the same conclusion follows directly from Lemma~\ref{lem:theta-independence}.

We now descend through the remaining powers of \(a\). Let \(e<D\), and suppose that all \(d_\lambda\) with \(D(\lambda)=D\) and \(E(\lambda)>e\) have already been eliminated. In the coefficient of \(a^eb^D\), terms whose leading \(b\)-degree is below \(D\) cannot contribute; the two possible \(p\)-terms have already been eliminated; terms with \(E(\lambda)<e\) cannot reach \(a\)-degree \(e\); and terms with \(E(\lambda)>e\) have zero coefficients by the induction hypothesis. Lemma~\ref{lem:Q-leading} therefore shows that this coefficient is exactly
\[
 \sum_{\substack{\lambda\in\Lambda\\D(\lambda)=D,\ E(\lambda)=e}}
 d_\lambda\Theta_\lambda(X).
\]
Lemma~\ref{lem:theta-independence} forces all displayed \(d_\lambda\) to vanish. Descending from \(e=D-1\) to \(e=0\) eliminates every normal-form term whose leading \(b\)-degree is \(D\).

Only finitely many coefficients occur in \eqref{eq:Phi-normal}. Repeating the argument at the next largest leading \(b\)-degree eliminates all \(c_r\) and \(d_\lambda\), contradicting the assumption that the normal form was nonzero. Hence \(\Phi\equiv0\).
\end{proof}

\section{Nonexistence of walk-realizable supporters}

\begin{corollary}\label{cor:no-supporter}
For no integer \(\ell\geq5\) does there exist a walk-realizable \(\mathcal F\)-supporter of order \(\ell\), where \(\mathcal F=\{2K_2,C_4,C_5\}\).
\end{corollary}

\begin{proof}
Suppose that \(\Phi(G)=\sum_{H\in\mathcal U_{\leq\ell}}d_H\ind(H,G)\) is such a supporter. If \(S\) is split, then every induced subgraph of \(S\) is split and hence \(\mathcal F\)-free. Definition~\ref{def:supporter} therefore gives \(\Phi(S)=0\). Since \(\Phi\) is walk-realizable, it belongs to \(\mathcal T\), and Theorem~\ref{thm:split-testing} yields \(\Phi\equiv0\).

On the other hand, for every \(F\in\mathcal F\), nonnegativity and \(\ind(F,F)=1\) give
\[
 \Phi(F)=\sum_Hd_H\ind(H,F)\geq d_F\ind(F,F)=d_F>0,
\]
a contradiction.
\end{proof}

Corollary~\ref{cor:no-supporter} proves
Theorem~\ref{thm:main}\textup{(ii)} and completes the resolution of
Conjecture~\ref{conj:WangTang}.

\begin{remark}\label{rem:stronger}
Theorem~\ref{thm:split-testing} is strictly stronger than
Corollary~\ref{cor:no-supporter}: it requires neither nonnegativity nor
an induced-subgraph expansion. It states that restriction to split
graphs is injective on the entire space \(\mathcal T\).

This obstruction is specific to finite linear combinations of the
patterned closed-walk counts \(n_\beta\). Thus it rules out the entire
finite-order supporter strategy for split graphs, rather than merely a
certificate of any particular order. It does not rule out certificates
arising from nonlinear functions of generalized spectral invariants,
infinite trace expansions, or invariants outside \(\mathcal T\).
\end{remark}

\bigskip
\noindent
Yanyang Li\\
School of Mathematics, Southeast University, Nanjing 211189, China\\
\textit{Email:} \texttt{liyanyang1219@gmail.com}

\end{document}